\documentclass[12pt,leqno,amscd,verbatim, url]{amsart}
\usepackage{amsfonts,latexsym}
\usepackage{amssymb,amsmath,amsthm,amscd,latexsym,amsfonts,bezier}
\usepackage{amsmath,amssymb}
\usepackage[latin1]{inputenc}
\usepackage{colortbl}
\usepackage[english]{babel}
\usepackage{amsmath}
\usepackage{amsfonts}
\usepackage{amssymb}
\usepackage{graphicx}%
\newtheorem{theorem}{Theorem}

\newtheorem{corollary}[theorem]{Corollary}

\newtheorem{lemma}[theorem]{Lemma}

\newtheorem{proposition}[theorem]{Proposition}
\newtheorem{remark}[theorem]{Remark}

\newcommand{\RInd}{\rm{RInd}\,}
\newcommand{\Res}{\rm{Res}\,}

\begin{document}

\dedicatory{\it We dedicate this paper to Michel Brou\'{e}.}
\title{Truncation and the Induction Theorem}
\author{Terrell L. Hodge}
\address{Department of Mathematics, Western Michigan University, Kalamazaoo, MI 49008}
\email{terrell.hodge@wmich.edu}
\author{Paramasamy Karuppuchamy}
\address{Department of Mathematics, University of Toledo, Toledo, OH 43606}
\email{paramasamy.karuppuchamy@utoledo.edu}
\author{Leonard L. Scott}
\address{Department of Mathematics, University of Virginia, Charlottesville, VA 22903}
\email{lls2l@virginia.edu}
\thanks{Research supported in part by NSF grant DMS-1001900 and Simons Foundation Collaborative
Research award $\#$359363.}
\thanks{This work, an accepted manuscript form, is licensed under the Creative Commons Attribution-NonCommercial-NoDerivatives 4.0 International License. To view a copy of this license, visit {\tt http://creativecommons.org/licenses/by-nc-nd/4.0/} or send a letter to Creative Commons, PO Box 1866, Mountain View, CA 94042, USA.  The associated published article in press is: T.L. Hodge et al., Truncation and the Induction Theorem, J. Algebra
(2020), {\tt https://doi.org/10.1016/j.jalgebra.2020.02.004}.}
\date{February 15, 2020}

\begin{abstract}
A key result in a 2004 paper by S. Arkhipov, R. Bezrukavnikov, and
V. Ginzburg compares the bounded derived category
$D^{b}(block(\mathbb{U}))$ of finite dimensional modules for the
principal block of a Lusztig quantum enveloping algebra
$\mathbb{U}$ at an $\ell^{\text{th}}$ root of unity with a special full subcategory $D_{triv}%
(\mathbb{B})$ of the bounded derived category of integrable type 1
modules for a Borel part $\mathbb{B}\subset\mathbb{U}.$
Specifically, according to this \textquotedblleft Induction
Theorem\textquotedblright\ \cite[Theorem
3.5.5]{ABG04} the right derived functor of induction $\mbox{Ind}_{\mathbb{B}%
}^{\mathbb{U}}\,$ yields an equivalence of categories $\mbox{RInd}_{\mathbb{B}%
}^{\mathbb{U}}\,:D_{triv}(\mathbb{B})\overset{\sim}{\rightarrow}%
D^{b}(block(\mathbb{U})).$ Some restrictions on $\ell$ are
required--e.g., $\ell >h$, the Coxeter number.  It is suggested
briefly \cite[Remark 3.5.6]{ABG04} that an analog of this
equivalence carries over to characteristic $p>0$ representations of
algebraic groups. Indeed, the authors of the present paper have
verified, in a separate preprint \cite{HKS16}, that there is such an
equivalence
$\mbox{RInd}_{B}^{G}:D_{triv}(B)\overset{\sim}{\rightarrow
}D^{b}(block(G))$ relating an analog of $D_{triv}(\mathbb{B}),$
defined using a Borel subgroup $B$ of a simply connected semisimple
algebraic group $G$, to the bounded derived category of the
principal block of finite dimensional rational $G$-modules. The
proof is not without difficulty and supplies new, previously missing
details even in the quantum case. The present paper continues the
study of the modular case, taking the derived category equivalence
as a starting point. The main result here is that, assuming
$p>2h-2$, the equivalence behaves well with respect to certain
weight poset ``truncations," making use of a variation by Woodcock
\cite{W97} on van der Kallen's \textquotedblleft excellent
order\textquotedblright\ \cite{vdK89}. This means, in particular,
that the equivalence can be reformulated in terms of derived
categories of finite dimensional quasi-hereditary algebras. We
expect that a similar result holds in the quantum case.

In \cite[Thm. 2.1]{R94} J. Rickard proves a theorem stating existence, in an
algebraic groups context, of some derived equivalences with specified corresponding
character isometries. In an appendix to this paper, we discuss a similar more recent result
\cite[Lem. 3.2]{HKS16}.  It gives additional information on the behavior of
the derived equivalences, giving their action on some natural
right-derived induced objects, lifting the given character isometries.

\end{abstract}
\maketitle


\section{Introduction}
Suppose $G$ is a semisimple, simply connected algebraic group
defined and split over $\mathbb{F}_p,$ with Coxeter number $h$ and
Borel subgroup $B.$ Take $\mathcal{C}_{G\text{ }}^{f}$ to be the
finite dimensional rational modules for $G,$ and within it,
$block(G)$, to be the principal block of $G.$

It had been known for some time that
$block(G)\subseteq\mathcal{C}_{G\text{ }}^{f}$ fully embeds, via
module restriction  $\Res^G_B$, into the category $\mathcal{C}_{B}^{f}$ of finite
dimensional rational $B-$modules. However, while there are some
abstract characterizations in \cite{PSW00}, there is no known
explicit description of 
 $\Res^G_B(block(G))$ in terms of
$B$-modules. The full embedding of $block(G)$ via $\Res^G_B$ even
 yields a full embedding of $D^b(block(G))$ in $D^b(\mathcal{C}_{B}^{f}),$ for the bounded derived categories
of $block(G)$ and $\mathcal{C}_{B}^{f}$ (\cite{CPSvdK77}); there
again, the question of how  to explicitly describe the image 
of $D^b(block(G))$ in $D^b(\mathcal{C}_B^{f})$ remains
open. 
Nevertheless, in contrast, an explicit description of $D^b(block(G))$ sitting inside 
$D^{b}(\mathcal{C}_{B\text{ }}^{f})$ does exist;  as recorded in Theorem  \ref{thm:pinduction} below  and first suggested by \cite{ABG04},  it is given by a triangulated subcategory $D_{triv}(B)$ of 
$D^{b}(\mathcal{C}_{B\text{ }}^{f})$. A
thematic feature of the approach taken in \cite{ABG04} for quantum groups, and adapted
to the characteristic $p$ algebraic groups case in \cite{HKS16},  is
to focus on the right derived functor $\RInd$ of induction, that is, to focus on the right adjoint of restriction, rather than restriction itself.

Specifically, within $D^{b}(\mathcal{C}_{B\text{ }}^{f}),$ take $D_{triv}(B)$ to be the full triangulated subcategory consisting
precisely of all objects having cohomology both of finite dimension and also
with weights all expressible in the form $p\lambda$ for some weight $\lambda$ in the root lattice. The following theorem is proved in
\cite{HKS16}.

\begin{theorem}  \label{thm:pinduction} 
Assume $p>h.$ Then the functor $\RInd_{B}^{G}\,$ induces an
equivalence of triangulated categories
\[
D_{triv}(B)\overset{\sim}{\longrightarrow}  D^{b}(block(G))\text{.}%
\]

\end{theorem}

In particular, $D^b(block(G))$ is the strict image of $D_{triv}(B)$ under the
induction functor $\mbox{Rind}_{B}^{G}.$
We mention that the functor giving the equivalence in Theorem \ref{thm:pinduction}
is
obtained from the induction functor $\mbox{RInd}_{B}^{G}:D^{b}(\mathcal{C}%
_{B\text{ }})\rightarrow D^{b}(\mathcal{C}_{G\text{ }})$ by restricting its
domain and range. \footnote { Here we have removed the superscript $\ f$ \ from
$\mathcal{C}_{B\text{ }}^{f}$ and $\mathcal{C}_{G}^{f}$ to indicate
categories of all (possibly infinite dimensional) rational $G-$modules, to
allow the standard construction of $\mbox{RInd}_{B}^{G}$ using complexes of
injective modules. }

\medskip

Arkhipov, Bezrukavnikov, and Ginzburg first studied such remarkable properties of induction in the context of a quantum enveloping algebra at an $\ell^{\text{th}}$ root of unity. The statement of Theorem 1 is modeled on their quantum ``Induction Theorem" \cite[Thm 3.5.5]{ABG04} by   replacing
the quantum enveloping algebras associated with the groups
$B$ and $G$ by the groups themselves, and
replacing $\ell$  by $p$. (There are several conditions on $\ell$ in the quantum case, but they reduce to the requirement $\ell>h$ when $\ell$ is a prime.)  It seems the authors of \cite{ABG04} believed 
Theorem \ref{thm:pinduction} to be true, and that some of the
ingredients of their proof of its quantum analog \cite[Thm 3.5.5]{ABG04} could be applied to its proof. This much, at least, is
confirmed by the  proof of Theorem \ref{thm:pinduction} in
\cite{HKS16}, though new ingredients beyond the proof of \cite[Thm
3.5.5]{ABG04} are also required.\footnote{It should also be mentioned
that the proof in \cite{ABG04} for the quantum case \cite[Thm
3.5.5]{ABG04} appears to be incomplete, in the sense that the proof
of one key lemma, \cite[Lem. 4.1.1(ii)]{ABG04}, seems inadequate.
Fortunately, the lengthy argument for a corresponding result
\cite[Lem. 3.2(ii)]{HKS16} applies in both the modular and quantum cases.
This lemma is stated in the appendix to the present paper.}

The central focus of the paper in hand is the following finiteness property, a collateral discovery from
 the 
process of finding a proof of Theorem
\ref{thm:pinduction}:  
When also $p\geq2h-2$ 
(a condition equivalent to $p>2h-2$ when $p>h$),
the equivalence in Theorem \ref{thm:pinduction}  can be written as a
union of equivalences of triangulated categories associated to
highest weight categories in the sense of \cite{CPS88}, each having
a finite weight poset.

To give some detail,
 for each positive integer $m,$ we will
set $\Lambda_{m}$ to be a particular subset in the root lattice
$\mathbb{Y}$ of $G;$ then let $\Gamma_m = \Lambda_m \cap (W_p\cdot
0)^+$ be those elements of $\Lambda_m$ that are among the dominant
weights $(W_p\cdot 0)^+$ indexing the irreducible modules in
$block(G).$ The set $\Gamma_m$ turns out to be a poset ideal in
$(W_p.0)^+$ with respect to the dominance order, and we let
$block(G)_{\Gamma_m}$ be the image of $block(G)$ under the
associated standard dominant weight poset truncation by $\Gamma_m.$
The subset $\Lambda_{m}$ turns out also to be a weight poset, but
with respect to an order associated to both the ``excellent order''
and ``antipodal excellent order"  highest weight categories of
$B$-modules \cite{W97}, with terminology introduced (with different
weight posets) by van der Kallen \cite{vdK89}, \cite{vdK93}. There
is a corresponding associated notion of poset truncation for the
distribution algebra $Dist(B)$ of $B,$ relative to $\Lambda_m.$
Using these notions, we can formally state the main result of this paper:

\begin{theorem} \label{thm:truncation1}
Under the assumption $p>2h-2,$ the equivalence $\RInd_{B}^{G}\,:$ $D_{triv}%
(B)\rightarrow D^{b}(block(G))$ of Theorem  \ref{thm:pinduction} induces, for each integer
$m>0$, equivalences of full triangulated subcategories%
\[
D_{triv}(Dist(B)_{\Lambda_{m}})\overset{\sim}{\longrightarrow} D^{b}(block(G)_{\Gamma_{m}}).
\]
Also, $D_{triv}(B)$ is naturally equivalent to the directed union of
its full triangulated subcategories
$D_{triv}(Dist(B)_{\Lambda_{m}}), \,m>0$ an integer, and
$D^{b}(block(G))$ is similarly naturally equivalent to the directed
union of the various $D^{b}(block(G)_{\Gamma_{m}}).$ 
\end{theorem}

\noindent We remark that for each positive integer $m,$ the categories $Dist(B)_{\Lambda_m}$ and $block(G)_{\Gamma_{m}}$ are
equivalent to categories of finite dimensional modules for finite dimensional quasi-hereditary algebras.
We believe this ``finiteness property'' of the equivalence in Theorem \ref{thm:pinduction}, as given by Theorem \ref{thm:truncation1},
is of sufficient significance to
be worth recording on its own.

\bigskip

The current version of \cite{HKS16}, a project of several years,  was posted in early 2016. A few weeks later,
at the AIM workshop, ``Sheaves and Representations of Reductive Groups,''  Geordie Williamson
took note of the paper as having ``revisited ABG'' in remarks to the workshop participants. He also remarked, in effect,  that the formal similarities of the paper's main result (see Theorem 1 above)   could be of interest to researchers studying  the Brou\'{e} abelian defect group conjecture in modular representations of finite groups. It was/is a reasonable comment, especially in view of \cite{ChR08}, which introduced an approach to derived category equivalences, proving with it both the symmetric groups case of the Brou\'{e} conjecture, and also proving, in the type A case, a conjecture  of Rickard for reductive algebraic groups. While this confluence was surprising at the time, Rickard does say  in his reductive groups paper \cite[p.256]{R94} that he was motivated by symmetric groups. He gave no obvious  parallel to the abelian defect group conjecture,  but did present a  conjecture on interesting derived equivalences which might occur, for all types, for reductive algebraic groups. He proved it  in the regular weight case \cite[Thm. 2.1]{R94}.  We observed some time after the AIM conference that  \cite[Lem. 3.2]{HKS16} already  fits very well into this discussion, providing additional information on the derived equivalences studied in \cite[Thm. 2.1]{R94}  by Rickard. We include a discussion of  \cite[Lem. 3.2]{HKS16}, stated in the algebraic groups case as Lemma \ref{wallcrossingfunctors} in an appendix to this paper, in the hope of sparking further interest in this line of investigation.
\bigskip

\section{Notation, Conventions, and Other Preliminaries}
\subsection{Basics}
We provide below the list of notation
 for this paper. Most of the less standard terminology is
 repeated in the main text.

\begin{itemize}
\item $p$~ a prime
\item $G$~ a semisimple, simply connected algebraic group defined and split over $\mathbb{F}_p$
\item $R$~ fixed root system for $G;$ with $R = R^- \cup R^+$ for a fixed set of negative (resp., positive) roots $R^-$ (resp., $R^+$)
\item $W$~ Weyl group of $G$ w.r.t. $R$ (more precisely, there is split maximal torus which determines both $W$ and $R$)
\item $W_p$~ (also denoted $W_{aff}$) the affine Weyl group $W_p \cong p\mathbb{Z}R \rtimes W$ of $G$ with respect to $R$
\item $h$~ Coxeter number for $G$
\item $B = B^{-} \subset G$~ Borel subgroup associated to $R^-$ (resp., $B^+$ positive Borel associated to $R^+$)
\item $\mathbb{X}$~ weight lattice for $G$
\item $\mathbb{X}^+$~ dominant weights, with the usual dominance order denoted  $\leq$, and also used on $\mathbb{X}$
\item $\mathbb{Y}$~ root lattice for $G$
\item $\Lambda_m$~ a subset of $\mathbb{Y}$ defined in Proposition 3
\item $\Gamma_m$~ a subset of  $\Lambda_m$ defined in Proposition 4 

\item $\rho = \frac{1}{2}\Sigma_{\alpha\in R^+}\alpha$ 
\item $w\cdot \lambda:= w(\lambda + \rho) - \rho,$~ defines the ``dot action'' of $W$ on $\lambda$ for $\lambda \in \mathbb{X}$
\item $w_o\in W,$~ the unique element of longest length in $W$
\item $x^{+}$ (resp., $x^{-}$)~ the unique dominant (resp., antidominant) weight in the $W$-orbit $Wx$ of $x\in \mathbb{X}$
\item $\preceq$~ a partial order defined on $\mathbb{X}$ by $x\preceq x^{\prime}$ if either the condition $x^{+}<x^{\prime+}$ holds (in the usual
dominance order), or else $x=x'$ and $w\leq w^{\prime}$~ (in the
Bruhat-Chevalley order), where $w,w^{\prime}$ $\in W$ are the unique
elements of minimum length with $x=wx^{-}$ and
$x^{\prime}=w^{\prime}{x^{\prime}}^{-} (=w^{\prime}x^{-})$
\item $x\preceq^{\circ}x^{\prime}$~ a partial order defined on $\mathbb{X}$ by $x\preceq^{\circ}x^{\prime}$ iff $w_{0}x\preceq w_{0}x^{\prime}$

\item  $\mathcal{C}_G$ (resp., $\mathcal{C}_B$)~ category of rational $G$-modules (resp., rational $B$-modules)
\item $\mathcal{C}^{0}_G\subseteq \mathcal{C}_G$~ full subcategory of rational
$G$-modules with all composition factors having highest weights in $(W_p \cdot 0)^+$~ 
\item $\mathcal{C}^{0}_B\subseteq \mathcal{C}_B$~ full subcategory of rational $B$-modules with weights in $\mathbb{Y}$
\item $\mathcal{C}^f_G$ (reps., $\mathcal{C}^f_B$)~ category of finite dimensional rational $G$-modules (resp., finite dimensional rational $B$-modules)
\item $\mathcal{C}^{0,f}(G) \subseteq \mathcal{C}^{f}_{G}$~ the full subcategory consisting of all finite dimensional rational $G-$modules in $\mathcal{C}^{0}_G$;~ by
definition, this is, the ``principal block" $block(G)$ of $G$, and so we also write $block(G)  =
\mathcal{C}^{0,f}(G)$
\item $\mathcal{C}^{0,f}(B) \subseteq \mathcal{C}^{f}_{B}$~ the full subcategory consisting of all finite dimensional rational $B-$modules
for which the weights are in $\mathbb{Y};$ thus $\mathcal{C}^{0,f}(B) = block(B),$ the principal block of $B$ (although we shall not explicitly use this
latter fact)
\item $D^b(\mathcal{A})$~ the bounded derived category of an appropriate category $\mathcal{A}$ 
\item $D_{triv}(B) \subset D^b(\mathcal{C}^{f}_B)$~ the full triangulated subcategory consisting
precisely of all objects having cohomology both of finite dimension and also
with weights all expressible in the form $p\lambda$ for some weight $\lambda$ in $\mathbb{Y}$

\item $\mathcal{C}_{B\text{ }}%
^{0,f}[\Lambda]\subseteq$~ \ $\mathcal{C}_{B\text{ }}^{0}$,~ the
full subcategory of all finite-dimensional objects whose weights all
belong to $\Lambda,$~ for any $\Lambda\subseteq \mathbb{Y}$ a finite
poset ideal (with respect to either order $\preceq$ or
$\preceq^{\circ}$)

\item $\mathcal{C}_{G}^{0,f}[\Gamma]$~ the full subcategory of $\mathcal{C}_{G}^{0,f}$ consisting of those
objects whose composition factors all have highest weights in $\Gamma,$ for $\Gamma\subseteq(W_{aff}\cdot0)^{+}=(W_{p}%
\cdot0)^{+}$~ a finite poset ideal under the dominance order.
\end{itemize}

\subsection{Highest Weight Categories, Poset Orders, and Truncation}
The category
$\mathcal{C}_{B\text{ }}$ of all rational $B-$modules 
is a highest weight category with respect to either the ``excellent"
or ``antipodal excellent" partial orders on weights in \cite{vdK89}. 
Following \cite{W97}, we will use variations, respectively denoted
here by $\preceq$ and $\preceq^{\circ}$, of these orders, which \
give essentially the same \footnote{This fact, implicitly suggested
in \cite{W97} and used in \cite{PSW00}, does not seem to be
explicitly proved in the literature. We intend to later include a proof, in a paper in preparation (with Brian Parshall and some or all of us) on quantum
analogs of the results in this paper. In the meantime, the reader
can simply rely on \cite{W97} for the fact that these orders each
define highest weight categories.} respective highest weight
category structures (give the same costandard and standard modules).
More precisely, for $x\in\mathbb{X,}$ let $x^{+}$ denote the unique
dominant member of its (undotted) Weyl group orbit, and let $x^{-}$
denote the unique antidominant member of this orbit. Then $x\preceq
x^{\prime}$ is defined, for $x,x^{\prime}\in\mathbb{X,}$ to mean
that either the condition $x^{+}<x^{\prime+}$ holds (in the usual
dominance order), or else $x^{+}=x^{\prime+}$ and $w\leq w^{\prime}$ (in the
Bruhat-Chevalley order), where $w,w^{\prime}$ $\in W$ are the unique
elements of minimum length with $x=wx^{-}$ and
$x^{\prime}=w^{\prime}x^{-}$. \ (Notice that $x^{+}=x^{\prime +}$ $\
$implies $x^{-}=x^{\prime-}$.) Define also
$x\preceq^{\circ}x^{\prime}$ iff
$w_{0}x\preceq w_{0}x^{\prime}$. In particular, the action of
$w_{0}$ interchanges $\preceq$ and $\preceq^{\circ}$.  Both orders
can be used for either $\mathcal{C}_{B\text{ }}$ or
$\mathcal{C}_{B^{+}\text{ }}$. The latter category was used in
\cite{vdK89}, but \cite{W97} uses the former, as we do
here.\footnote{It is instructive to note that \cite{W97} refers to
$\preceq$
 as the ``antipodal excellent order" and to
$\preceq^{\circ}$ as the ``excellent" order, terminology choices
which readers familiar with \cite{vdK93} might expect to have been
reversed. The apparent explanation is that \cite{W97} prefers $B$ (which has ``negative" root groups) to
$B^{+}$. Conjugation by $w_0$ carries $B$ to $B^{+}$ and $\preceq$
to $\preceq^{\circ}$.  Thus, Woodcock does appear to be trying to align
\cite{W97} with \cite{vdK93}.}
 In fact, we use
$\mathcal{C}_{B\text{ }}^{0}$, which inherits a highest weight
category structure from $\mathcal{C}_{B\text{ }}$.

Similarly, if
$\Lambda\subseteq \mathbb{Y}$ is a finite poset ideal (with respect
to
either order), then the full subcategory $\mathcal{C}_{B\text{ }}%
^{0,f}[\Lambda]\subseteq$ \ $\mathcal{C}_{B\text{ }}^{0}$, of all
finite-dimensional objects whose weights all belong to $\Lambda,$ inherits a
highest weight category structure \cite{CPS88}.

\begin{proposition}
Let $m$ be any positive real number (we will just use the integer
case), and put $\Lambda_{m}=\{y\in \mathbb{Y}|$
$|(y,\alpha^{\vee})|\leq mp$ for all $\alpha\in R^{+}\}.$ Then
$\Lambda_{m}$ is a poset ideal in $\mathbb{Y}$ with respect to
either of the orders $\preceq$ or $\preceq^{\circ}$.
\end{proposition}

\begin{proof}
First, note that $R^{+}$ can be replaced by $R=R^{+}\cup-R^{+}$ in
the definition of $\Lambda_{m}$; consequently, the latter set is
stable under the action of $W$. $\ $In particular, it is stable
under $w_{0}$, so it suffices to treat the order $\preceq$ . Also,
the stability implies, for $y\in \mathbb{Y,}$ that $y\in\Lambda_{m}$
iff $y^{+}\in\Lambda_{m}$. The latter holds iff 
$(y^{+}\alpha^{\vee})\leq mp$\ $\ $\ for all $\alpha\in R^{+}$,
which holds iff $(y^{+}\alpha_{0}^{\vee})\leq mp\,$, where
$\alpha_{0}$ denotes the maximal short root. Let $y\preceq
y^{\prime}$ with $y,y^{\prime}\in\mathbb{Y}$ and
$y^{\prime}\in\Lambda.$ Then $y^{+}\leq y^{\prime+}$ in the
dominance
order, which implies $(y^{+},\alpha_{0}^{\vee})\leq(y^{\prime+},\alpha_{0}%
^{\vee})\leq mp.$ Hence, $y\in \Lambda_{m}$, and the proposition is
proved.
\end{proof}

\bigskip

There is an easy analog of the proposition for dominant weights. We will just
use those in the weight poset $(W_{aff}\cdot0)^{+}$ of dominant weights in
$W_{aff}\cdot0$. These are the dominant weights which occur as highest weights
for irreducible modules in $block(G)$. We alert the reader that we will later
write $(W_{aff}\cdot0)^{+}=(W_{p}\cdot0)^{+}.$ We record the following result,
whose proof is immediate.

\begin{proposition}
Let $\ \ m$ be a positive real number (as in the previous
proposition), and put\ $\Gamma_{m}=\{y\in(W_{aff}\cdot0)^{+}|$
$(y,\alpha^{\vee})\leq mp$ for all
$\alpha\in R^{+}\}.$ Then $\Gamma_{m}$ is a poset ideal\ in $(W_{aff}%
\cdot0)^{+}$ with respect to the dominance order, and $\Gamma_{m}=$
$\Lambda_{m}\cap(W_{aff}\cdot0)^{+}$.
\end{proposition}

There are similar (easier) truncations for $\mathcal{C}_{G\text{
}}^{0,f}$ $=block(G).$ The category $\mathcal{C}_{G\text{ }}^{0}$is
a highest weight category with respect to the poset
$(W_{aff}\cdot0)^{+}$ of dominant weights in $W_{aff}\cdot0$ using
several orders, all of them equivalent in the sense of giving the same
costandard and standard modules. See \cite[1.5,1.8]{Ja03}. We will
just use the dominance order. We take this opportunity to note that
$W_{aff}$, in its affine action on $\mathbb{X}$, is denoted $W_{p}$ in
\cite[1.5,1.8]{Ja03}, with the $p$ reminding us that $W_{aff}$ acts
on $\mathbb{X}$ as the semidirect product of $W$ and $p\mathbb{Y}$, with $W$ acting linearly (before the
``dot" is introduced), and
elements of $p\mathbb{Y}$ acting by translation. If $\Gamma\subseteq(W_{aff}\cdot0)^{+}=(W_{p}%
\cdot0)^{+}$ is a finite poset ideal, let $\mathcal{C}_{G}^{0,f}[\Gamma]$
denote the full subcategory of $\mathcal{C}_{G}^{0,f}$ consisting of those
objects whose composition factors all have highest weights in $\Gamma.$ Then
$\mathcal{C}_{G}^{0,f}[\Gamma]$ inherits a highest weight category structure
from $\mathcal{C}_{G}^{0}.$

\section{Main Results: Bounded Derived Categories, Truncation, and the Induction Theorem} \label{sec:main}

Because $\mathcal{C}_{G}^{0,f}[\Gamma]$ above inherits a highest weight
category structure from $\mathcal{C}_{G}^{0},$ there is a natural
full embedding of triangulated categories
$D^{b}(\mathcal{C}_{G}^{0,f}[\Gamma])\subseteq
D^{b}(\mathcal{C}_{G}^{0,f})=D^{b}(block(G))$. To simplify notation,
we write
$D^{b}(\mathcal{C}_{G}^{0,f}[\Gamma])=D^{b,f}(Dist(G)_{\Gamma})$.\footnote{This
notation, in addition, suggests the (correct) fact that
$\mathcal{C}_{G\text{ }}^{0,f}[\Gamma]$ is naturally equivalent to
the category of finite dimensional modules for $Dist(G)_{\Gamma}$.
The latter is a (finite dimensional) quasi-hereditary algebra,
defined as the quotient of $Dist(G)$ by the ideal which is the
annihilator of all rational $G-$modules whose composition factors
have high weights only in $\Gamma$.
\par
Similar remarks may be made regarding the notation
$D^{b,f}(Dist(B)_{\Lambda })$ in the next paragraph. We leave the
fairly routine proofs in both cases to the interested reader. These
identifications, though informative, are not required for our main
results.} Thus, the previous\ strict full embedding is now written
$D^{b,f}(Dist(G)_\Gamma)\subseteq D^b(block(G)).$ With abuse of
notation, we will sometimes identify $D^{b,f}(Dist(G)_\Gamma)$ with
its strict image in $D^b(block(G))$.

Similarly, if
$\Lambda\subseteq\mathbb{Y}$ is a finite poset ideal with respect to
either $\preceq$ or $\preceq^{\circ}$, there is a natural full
embedding $D^{b}(\mathcal{C}_{B}^{0,f}[\Lambda])\subseteq D^{b}(\mathcal{C}%
_{B\text{ }}^{0,f})$. We write $D^{b}(\mathcal{C}_{B}^{0,f}[\Lambda
])=D^{b,f}(Dist(B)_{\Lambda})$ and also let $D_{triv}(Dist(B)_{\Lambda})$
denote the full subcategory of $D^{b,f}(Dist(B)_{\Lambda})$ whose cohomology
has only weights $py$ with $y\in \mathbb Y$ and $py\in\Lambda$. Then the full
embedding $D^{b}(\mathcal{C}_{B}^{0,f}[\Lambda])\subseteq D^{b}(\mathcal{C}%
_{B\text{ }}^{0,f})$ gives a full embedding
$D_{triv}(Dist(B)_{\Lambda })\subseteq D_{triv}(B).$ The strict image$\footnote{There seems to be no standard
terminology here. If $F$:$\mathcal{A}\rightarrow \mathcal{B}$ is a full
(triangulated) functor between (triangulated) categories, we define
the {\bf strict image} of $F$, or of $\mathcal{A}$ under $F$, to be the
smallest full (triangulated) subcategory of $\mathcal{B}$ containing,
for each object $X$ in $\mathcal{A}$, each object of $\mathcal{B}$ that is isomorphic
to $F(X)$.}$ of $D_{triv}(Dist(B)_\Lambda)$ is the (full)
subcategory, call it  $D_{triv,\Lambda}(B)$, of objects in $D_{triv}(B)$ represented by complexes
which have cohomology with all high weights in $\Lambda$. This
subcategory $D_{triv,\Lambda}(B)$ of $D_{triv}(B)$ is certainly interesting in its own
right, and its interpretation here as a strict image gives a
(non-obvious) way of viewing it inside the more ``finite"
$D^{b,f}(Dist(B)_\Lambda).$

Next, it makes sense to ask when the induction equivalence $\RInd_{B}
^{G}:$ $D_{triv}(B)\rightarrow D^{b}(block(G))$ takes $D_{triv}(Dist(B)_{\Lambda
})$ into $D^{b,f}(Dist(G)_{\Gamma}),$ and when the strict image of
$D_{triv}(Dist(B)_{\Lambda})$ $\ $contains $D^{b,f}(Dist(G)_{\Gamma}).$ There
are easy combinatorial sufficient conditions in each case. Let $\mathbb{X}%
^{+}$ $\subseteq\mathbb{X}$ denote the set of dominant integral weights.

\begin{proposition}\label{prop5} 
\bigskip Let   $\Gamma\subseteq(W_{aff}\cdot0)^{+}$ and $\Lambda\subseteq\mathbb{Y}$ be finite poset ideals, as above.

(1) If $(W\cdot(\Lambda\cap
p\mathbb{Y}))\cap\mathbb{X}^{+}\subseteq\Gamma,$ then
$\RInd_{B}^{G}$ takes $D_{triv}(Dist(B)_{\Lambda})$ into $D^{b,f}%
(Dist(G)_{\Gamma})\subseteq D^{b}(block(G))$.

(2) If $W\cdot\Gamma\cap p\mathbb{Y}\subseteq\Lambda$, then the strict image of
$\RInd_{B}^{G}D_{triv}(Dist(B)_{\Lambda})$ $\subseteq D^{b}(block(G))$
contains $D^{b,f}(Dist(G)_{\Gamma})$.
\end{proposition}

\begin{proof}
Consider a complex $M$ representing an object $[M]$ in $D_{triv}%
(Dist(B)_{\Lambda})\subseteq D_{triv}(B)$. To prove
$\RInd_{B}^{G}$ $[M]$ belongs to
$D^{b,f}(Dist(G)_{\Gamma})\,,$ it suffices, by standard truncation
methods\ \cite{BBD82}, to take $M$ concentrated in a single
cohomological degree, which may be taken to be $0$.  Thus, $M$ has a finite filtration with sections
one dimensional\ $B-$modules, each identified with a weight
$py\in\Lambda\cap p\mathbb{Y.}$ Without loss, $M$ is itself one
dimensional, identifying with such a $py$. As is well known, there
is a unique dominant weight $\gamma$ in $W\cdot py$ (a verification
is included in the proof of Lemma \ref{lem6} below), 
which must belong to
$\Gamma$ when the hypothesis of (1) holds.   By Andersen's strong
linkage theory (available in \cite{Ja03}), each composition factor
of any cohomology group of $\RInd_{B}^{G}$ $[M]$ must then
have highest weight $\lambda\leq\gamma$ (even in the strong linkage
order). Thus, $\lambda$ belongs to the poset ideal $\Gamma.$ This
proves assertion (1).

Andersen's strong linkage theory also guarantees, under the hypothesis that $\gamma$ is dominant
and $\gamma=w\cdot py$ with $py\in\Lambda\cap p\mathbb{Y,}$ that the
irreducible module $L(\gamma)$ appears with multiplicity $1$ in the cohomology
of $\mbox{RInd}_{B}^{G}$ $[py],$ again identifying $py$ with the associated
one dimensional $B$-module. \ This gives an easy proof of assertion (2) by
induction: Suppose the hypothesis of (2) holds, and all irreducible modules 
$L(\gamma^{\prime})$ belong to the strict image of $\mbox{RIndf}_{B}%
^{G}D_{triv}(Dist(B)_{\Lambda})$ whenever $\gamma^{\prime}<\gamma.$ By definition,
$\mbox{RInd}_{B}^{G}$ $[py]$ belongs to $\mbox{RInd}_{B}^{G}D_{triv}%
(Dist(B)_{\Lambda})$ and, as remarked, has $L(\gamma)$ as a composition factor
of its cohomology with multiplicity $1$. All other composition factors
$L(\gamma^{\prime})$ satisfy $\gamma^{\prime}<\gamma$ and so belong to the
strict image \ of $\ \mbox{RInd}_{B}^{G}D_{triv}(Dist(B)_{\Lambda}).$ It then
follows that $L(\gamma)$ belongs to this strict image (which is a full triangulated subcategory of $D^{b,f}(Dist(G))$).
Now  (2) follows by induction. This proves the proposition.
\end{proof}

\begin{lemma} \label{lem6} 
\bigskip Let $m$ be a positive real number, and put $\Lambda=\Lambda_{m}$,
$\Gamma=\Gamma_{m}.$ Recall our standing hypothesis $p>h.$ Then both of the
following hold.

(1) \ The sets $\Lambda,\Gamma$ satisfy the hypothesis (and conclusion) of part
(1) of Proposition \ref{prop5}. 

(2) \ If, in addition,$\ m$ is an integer and $p\geq2h-2$, the sets
$\Lambda,\Gamma$ satisfy the hypothesis (and conclusion) of part (2) of Proposition \ref{prop5}. 
\end{lemma}

\begin{proof}
First, we need a claim (which\ does not involve $m$ and assumes only
$p\geq h $ ). Let $y\in\mathbb{Y}$, and put $\nu=py\,$. Let $w\in W$
with $w\cdot\nu$ +$\rho$ $\ $dominant. (At least one such $w$ always
exists; and the dominant weight $w\cdot\nu$+$\rho$ =$w(\nu+\rho)\
$is the unique dominant weight in $W(\nu+\rho)$. 

\underline{We
claim} $w$ is unique, and the weights $w\nu$\ and $w\cdot\nu$ are
also dominant. In addition $w$ is also the unique element in $W$
with $w.\nu$ dominant.

To prove the claim, note first that $0<|(\rho,\alpha^{\vee})|\leq
h-1<p$ for every root $\alpha$.  Since
$(w\rho,\alpha^{\vee})=(\rho,(w^{-1}\alpha)^{\vee}),$ it follows
that the dominant weight $w(\nu+\rho)=p(wy)+w\rho$ is regular: That
is, $(p(wy)+w\rho$, $\alpha^{\vee})\neq0$ for all roots $\alpha$. So
$w(\nu+\rho)$ has trivial stabilizer in $W$. Thus, $w$ is unique.
Next, for any simple root $\alpha,$ apply $(-,\alpha^{\vee})$ to
both sides of the equation
$$w\cdot\nu+\rho=p(wy)+w\rho.$$
If we ever had $(wy,\alpha^{\vee})<0$, then, using the bound
$|(w\rho,\alpha^{\vee})|<p$, the above equation would give
$(w\cdot\nu +\rho,\alpha^{\vee})<0,$ contradicting the dominance of
$w\cdot\nu+\rho$. Consequently, $(wy,\alpha^{\vee})\geq0$ for all
simple roots $\alpha,$ and so $wy$ must be dominant. Also,
$w\nu=pwy$ must be dominant. Next, since $w\cdot\nu$+$\rho$
=$w(\nu+\rho)$ is both dominant and regular (in the sense above),
the weight$\ w\cdot\nu$ must be dominant. (In particular this gives
the verification promised in the proof of Proposition \ref{prop5}.) 
Finally, if $w^{\prime}\in W$ is such that $w^{\prime}\cdot\nu$ is
dominant, then $w^{\prime}(\nu+\rho)=w^{\prime }\cdot\nu+\rho$ is
also dominant. Thus, $w^{\prime}(\nu+\rho)=w(\nu+\rho ),$which gives
$w=w^{\prime}$ by regularity of $w(\nu+\rho)$, as previously noted.
This completes the proof of the claim.

Note that, as a consequence, one can deduce from dominance of $w\cdot\nu$, when
$\nu\in p\mathbb{Y}$ and $w\in W$, that $w\nu$ is dominant. We will often
use below the claim in this way.

Next, suppose $\nu$ belongs to $\Lambda$, as well as\ to
$p\mathbb{Y}$, and suppose $w\in W$ is such that $w\cdot\nu$ is
dominant. Thus, $w\nu$ is also dominant, by the claim. From the
inequality $w\cdot\nu=w\nu+w\rho-\rho\leq w\nu$, we obtain, for each
positive root $\alpha,$ the $\ $inequality
$(w\cdot\nu,\alpha^{\vee})\leq(w\cdot\nu,\alpha_{0}^{\vee})\leq(w\nu
,\alpha_{0}^{\vee})=|(\nu,w^{-1}\alpha_{0}^{\vee})|\leq mp$. \
(Recall the
definition of $\Lambda=\Lambda_{m}.$) Thus $w\cdot\nu\in$ $\Lambda_{m}%
\cap(W_{p}\cdot0)^{+}=$ $\Gamma_{m}=\Gamma.$ We have shown $W\cdot(\Lambda\cap
p\mathbb{Y)}\cap\mathbb{X}^{+}\subseteq\Gamma,$ the hypothesis of part (1) of
Proposition \ref{prop5}. 
Part (1) of the lemma follows.

Finally, suppose $\nu$ belongs to $W\cdot\Gamma$, as well as to
$p\mathbb{Y}$, and let $w\in W$ be such that $w\cdot\nu\in\Gamma.$
We want to show $\nu \in\Lambda,$ as required in part (2) of the
lemma. Assume $m$ is a (positive) integer and that $p\geq2h-2\,$, as
given in the hypothesis of part (2). Note that, since also $p>h\geq
2,$ we have the strict inequality $p>2h-2.$ Write $\nu=py$ with
$y\in\mathbb{Y.}$ By the claim, the weight $w\nu$ is dominant. To
prove $\nu\in\Lambda$=$\Lambda_{m}$ we just need to show
$|(\nu,\alpha^{\vee})|\leq mp$ for all roots $\alpha$. But
$|(\nu,\alpha^{\vee})|=|(w\nu,w\alpha^{\vee})|=|(w\nu,\pm w\alpha^{\vee}%
)|\leq(w\nu,\alpha_{0}^{\vee})\,$, the last inequality following from the
dominance of $w\nu$. $\ $\ Also, $\ $we have $(w\nu,\alpha_{0}^{\vee
})=(wpy,\alpha_{0}^{\vee})=p(wy,\alpha_{0}^{\vee})$ and
$$p(wy,\alpha_{0}^{\vee})
 =(w\nu,\alpha_{0}^{\vee})=(w\cdot\nu,\alpha
_{0}^{\vee})+(\rho-w\rho,\alpha_{0}^{\vee})
\leq mp+2h-2 < (m+1)p.$$
Since $m$ and $(wy,\alpha_{0}^{\vee})$ are integers, we have $(wy,\alpha
_{0}^{\vee})\leq m,$ and so $(w\nu,\alpha_{0}^{\vee})=p(wy,\alpha_{0}^{\vee
})\leq pm$. \ Thus $\nu\in\Lambda_{m}=\Lambda$, as required. The hypothesis
$W\cdot\Gamma\cap p\mathbb{Y}\subseteq\Lambda$ of part (2) of 
Proposition \ref{prop5} 
 is now verified, and so both parts of the lemma have been proved.
\end{proof}

\bigskip

The hypothesis on $p$ \ in the purely combinatorial result below is
essentially the same as $p\geq2h-2$, given our standing hypothesis $p>h.$
There are no other hypotheses, except for the stated one on $m.$

\begin{corollary}
Let $m$ be a positive integer and assume $p>2h-2.$ Then there is a 1-1
correspondence between\ $\Lambda_{m}\cap p\mathbb{Y}$ and $\Gamma_{m}$ , given by
sending an element $\nu\in\Lambda_{m}\cap p\mathbb{Y}$ to the unique dominant weight
$\gamma$ in $W\cdot\nu.$ This weight $\gamma$ is in $\Gamma_{m}$. In the inverse
direction, a weight $\gamma\in\Gamma_{m}$ \ is sent to the unique weight $\nu$ in
$W\cdot\gamma$ of the form $py$ with $y\in\mathbb{Y}.$ This weight $\nu$ is in
$\Lambda_{m}\cap p\mathbb{Y}.$
\end{corollary}

\begin{proof}
We will use Lemma \ref{lem6} and also quote the claim in its proof.  If
$\nu\in\Lambda_{m}\cap p\mathbb{Y},$ then part (1) of Lemma \ref{lem6}
implies $W\cdot\nu\cap \mathbb{X}^{+}\subseteq\Gamma_{m}.$ In addition, the claim in
the proof shows there is only one $v\in W$ with $v\cdot\nu\in \mathbb{X}^{+}.$ Thus
$\nu\mapsto v\cdot\nu$ gives a well-defined map $\Lambda_{m}\cap p\mathbb{Y}\rightarrow
\Gamma_{m}$.  

 Next, if we start with a $\gamma\in\Gamma_{m}$, then part (2)
implies $W\cdot\gamma\cap p\mathbb{Y\subseteq}\Lambda_{m}\cap p\mathbb{Y}$. If
we have $w\cdot\gamma=py$ and $w^{\prime}\cdot\gamma=py^{\prime}$ for some
$w,w^{\prime}\in W$ and $y,y^{\prime}\in \mathbb{Y}$, then regularity of the action of
$W_{p}$ on $W_{p}\cdot0$ forces $w=w^{\prime}$ and $y=y^{\prime}$. In
particular, the assignment $\gamma\mapsto\nu=w\cdot\gamma=py$ gives a
well-defined map $\Gamma_{m}\rightarrow\Lambda_{m}\cap p\mathbb{Y}$.

 Since $\gamma=w^{-1}\cdot\nu$, the construction above of the map $\Lambda_{m}\cap p\mathbb{Y}\rightarrow
\Gamma_{m}$ shows it sends $\nu$ to $\gamma$.

That is, the map $\Gamma_{m}\rightarrow\Lambda_{m}\cap p\mathbb{Y}\rightarrow
\Gamma_{m}$ is the identity.  Similarly, 
 the composite $\Lambda_{m}\cap p\mathbb{Y}\rightarrow
\Gamma_{m}\rightarrow\Lambda_{m}\cap p\mathbb{Y}$ is also the identity, and
the proof is complete.
\end{proof}

\begin{remark}
So far, we have not used Theorem 1 in this section. 
It is easy to deduce, either from Theorem 1 or from  the above three results of this section, that the functor
$\RInd_{B}^{G}\,$ induces an isomorphism of Grothendieck groups
\[
K_{0}(D_{triv}(B))\overset {\sim} {\longrightarrow}{K_{0}(D^{b}(block(G))}\text{.}
\]
This isomorphism may be regarded as a ``shadow" of Theorem 1, though
with a more restrictive bound $p>2h-2$ on $p$. \end{remark}  

We do not know a proof of Theorem 1 based on the isomorphism of Grothendieck groups above, even with the stricter bound, though it remains natural to look for
such an argument.   Instead, we now quote Theorem 1, together with the lemma
above, $\bf{to~prove~Theorem~2}$ , our main result.

\begin{proof}
By part (1) of Proposition \ref{prop5}
and part (1) of Lemma \ref{lem6}, 
 if
$N\in D_{triv}(B)$ is isomorphic to an object in $D_{triv}(Dist(B)_{\Lambda
_{m}})\subseteq D_{triv}(B)$, then $\mbox{RInd}_{B}^{G}\,N$ is isomorphic to
an object in $D^{b,f}(Dist(G)_{\Gamma_{m}})\subseteq D^{b}(block(G))$. Also,
by part (2), the full triangulated subcategory generated by all objects
$\mbox{RInd}_{B}^{G}\,N$ contains an isomorphic copy of each object in
$D^{b,f}(Dist(G)_{\Gamma_{m}})$. However, since $\mbox{RInd}_{B}^{G}\,:$
$D_{triv}(B)\rightarrow D^{b}(block(G))$ $\ $is an equivalence, the
collection\ \ $\mathcal{E}_{m}$ of all these objects $\mbox{RInd}_{B}^{G}\,N$
is already a full triangulated subcategory of $D^{b}(block(G))$, equivalent to
$D_{triv}(Dist(B)_{\Lambda_{m}})$. Since $\mathcal{E}_{m}\subseteq
D^{b,f}(Dist(G)_{\Gamma_{m}})$, up to isomorphism of objects, and every object
of $D^{b,f}(Dist(G)_{\Gamma_{m}})$ is isomorphic to an object in
$\mathcal{E}_{m},$ the inclusion of $\mathcal{E}_{m}$ in the full subcategory
$\mathcal{F}_{m}$ of objects in $D^{b}(block(G))$ isomorphic to an object in
(the image in $D^{b}(block(G))$) of $D^{b,f}(Dist(G)_{\Gamma_{m}})$ is an
equivalence of triangulated categories. So is the natural functor
$D^{b,f}(Dist(G)_{\Gamma_{m}})$ $\rightarrow\mathcal{F}_{m}.$ \ To summarize,
the functor $\mbox{RInd}_{B}^{G}$ directly induces an equivalence
$D_{triv}(Dist(B)_{\Lambda_{m}})\rightarrow\mathcal{E}_{m}$, while the
latter triangulated category is equivalent to $D^{b,f}(Dist(G)_{\Gamma_{m}})$
through the composite of the equivalence $\mathcal{E}_{m}\subseteq
\mathcal{F}_{m}$ and an inverse for the equivalence $D^{b,f}(Dist(G)_{\Gamma
_{m}})$ $\rightarrow\mathcal{F}_{m}$. \ This proves the first assertion of Theorem 2.   

The second assertion, regarding directed unions, follows from general
derived category ``recollement" considerations in highest weight
category theory \cite{CPS88}, together with the obvious facts that
\[
\mathbb{Y=}%
{\textstyle\bigcup\limits_{m>0}}
\Lambda_{m}\text{ and }W_{p}\cdot0=%
{\textstyle\bigcup\limits_{m>0}}
\Gamma_{m}%
\]
with the subscripts $m$ always taken to be positive integers. This
completes the proof of Theorem 2. 
\end{proof}

\section{Appendix}

 The result stated below is \cite[Lem. 3.2]{HKS16} in the algebraic groups case.
Thus $G$ is a semisimple and simply connected algebraic group over an
algebraically closed field of positive characteristic $p$. The paper \cite{HKS16} makes a blanket assumption
that $p>h$, though it appears $p\geq h$ is sufficient for this lemma.
  
 We continue  with details of the lemma's formulation, assuming $p\geq h$. This condition implies
that the weight $0$ is $p$-regular, and that the principal block, denoted $block(G)$, has as
its irreducible modules those with highest weights (dominant and) in the orbit $W_{aff}\cdot 0$.
 See  $\S2.1$ for notation.  We also denote here $(w\cdot 0)^{y}:=wy\cdot 0$, for any $y,w\in W_{aff}$.
Finally, just as in Rickard's paper \cite{R94}, we need appropriate ``translation'' functors.
See \cite{Ja03}, especially \cite[II, Lem. 7.6]{Ja03}.
Let $\Xi_\alpha:block(G)\longrightarrow block(G)$ denote
the composition of  two translation functors, first ``to the wall'' labeled
by a simple reflection $s_\alpha$ and, then, back ``out from
the wall''. There are canonical adjunction morphisms
$f:id\longrightarrow\Xi_\alpha$,and $g:\Xi_\alpha\longrightarrow
id$. The mapping cone $C(f)$ of $f$
gives rise to a triangulated functor from $D^{b}(block(G))$ to
itself, denoted $\theta^{+}_\alpha$. A similar construction (using a shifted mapping cone
$C(g)[-1]$) gives a functor $\theta^{-}_\alpha$.

The result also has a natural formulation identical to \cite[Lem. 4.1.1(i)]{ABG04} in the quantum case, on which
the statement in the algebraic groups case was modeled. The version \cite[Lem. 3.2]{HKS16} treats
both cases simultaneously.

\begin{lemma}\label{wallcrossingfunctors} With the notation discussed above, we have \\

(i) In $D^{b}(block(G))$ we have the following canonical
isopmorphisms $ \theta_\alpha^+\circ \theta_\alpha^- \cong id ~~
{\rm and }~~\theta_\alpha^-\circ \theta_\alpha^+ \cong id$. In
particular $\theta _\alpha ^+$ and $\theta_\alpha^-$ are
autoequivalences.

(ii) If $\lambda \in W_{aff} \cdot 0$ and $\lambda
^{s_\alpha}>\lambda$ then $\theta ^+_\alpha(\RInd_B^G\lambda)\cong
\RInd_B^G (\lambda^{s_\alpha})$.
\end{lemma}

 The proof of part (i) may be obtained from
the argument for \cite[Thm. 2.1]{R94}, or, alternately, from the
argument for \cite[Lem. 4.1.1(i)]{ABG04}. (Both arguments involve
similar ingredients.) The proof of part (ii) is given in
appendices A and B of \cite{HKS16} by an argument treating both the algebraic groups
and quantum cases simultaneously.

\medskip 
 In the algebraic groups case, Rickard's result
\cite[Thm. 2.1]{R94} is quite similar to part (i) of Lemma \ref{wallcrossingfunctors} above, though his result contains
additionally an explicit description of the effect of his derived equivalences on characters.  (Also, it is formulated for reductive,
rather than semisimple groups, and  for all regular weights, not just $W_{aff}\cdot 0$. However, translation functor arguments
render the two contexts essentially equivalent.)  We view  part (ii) of Lemma 9 as giving  
further information about  the derived equivalences of part (i), with an explicit character isometry,
agreeing with that in \cite[Thm. 2.1]{R94}, obtainable by taking Euler characteristics.

\section{Acknowledgements}

   This research was supported in part by National Science Foundation grant DMS-1001900 and Simons Foundation Collaborative
Research award $\#$359363. The authors thank the referee for a careful reading of the paper and helpful remarks.


\begin{thebibliography}{00}                                                                                  %


\bibitem{ABG04}S. Arkhipov, R. Bezrukavnikov, V. Ginzburg,
\emph{Quantum groups, the loop Grassmannian, and the Springer resolution,} J.
Amer. Math. Soc. \textbf{17} (2004), 595--678 (electronic version April 13, 2004).

\bibitem{BBD82}A. Beilinson, J. Bernstein, P. Deligne, \emph{Analyse et
topologie sur les espaces singuliers (1)}, Asterisque 100 (1982), 172pp.

\bibitem{ChR08} J. Chuang, R. Rouquier, \emph{Derived equivalences for symmetric groups
and $sl_2$-categorification}, Ann. of Math. (2) 167, no. 1 (2008), 245-298.

\bibitem{CPS88}E. Cline, B. Parshall, and L. Scott, \emph{Finite
dimensional algebras and highest weight categories}, J. reine u. angew. Math.
391 (1988), 85-99..

\bibitem{CPSvdK77}E. Cline, B. Parshall, L. Scott, and W. van der
Kallen, \emph{Rational and generic cohomology}, Invent. math. 39 (1977), 143-169.

\bibitem{HKS16}T. Hodge, P. Karuppuchamy , L. Scott, \emph{Remarks on
the ABG induction theorem}, preprint arXiv 1603.05699, 46pp.

\bibitem{Ja03}J. C. Jantzen, \textit{Representations of algebraic
groups, 2nd ed.}, American Mathematical Society (2003).

\bibitem{PSW00}B. Parshall, L. Scott, J.-p. Wang, \emph{Borel
subalgebras redux with examples from algebraic and quantum groups}, Algebras
and Rep. Theory \textbf{3} (2000), 213--257.

\bibitem{R94} J. Rickard, Translation Functors and
Equivalences of Derived Categories for Blocks of Algebraic Groups,
{\em Finite Dimensional Algebras and Related Topics}, Kluwer
Academic Publishers (1994), 255--264.

\bibitem{vdK89}W. van der Kallen, \emph{Longest weight vectors and
excellent filtrations}, Math. Zeit. \textbf{201} (1989), 19--31.

\bibitem{vdK93}W. van der Kallen, \emph{Lectures on Frobenius
Splittings and $B$-modules (Notes by S. P. Inamdar)}, Tata Institute
Notes/Springer.



\bibitem{W97} D. Woodcock, Schur Algebras and Global Bases: New Proofs of
Old Vanishing Theorems, {\it Journal of Algebra} {\bf 191} (1997) 331-370. 



\end{thebibliography}
\end{document}